\documentclass[preprint,11pt]{elsarticle}
\topskip  \usepackage{amssymb}
\usepackage{amsmath}
\usepackage{amsthm}
\usepackage{hyperref}
\usepackage{rotating}
\newtheorem{definition}{Definition}
\newtheorem{theorem}{Theorem}

\newtheorem{corollary}{Corollary}

\numberwithin{equation}{section} %To enable multi-level eguation numbering
\numberwithin{definition}{section}
\numberwithin{theorem}{section}
\numberwithin{lemma}{section}
\numberwithin{corollary}{section}
\numberwithin{remark}{section}

\begin{document}
\title{Multiparameter poly-Cauchy and poly-Bernoulli numbers and polynomials}
\author[Desouky]{B. S. El-Desouky \corref{cor2}}
\author[Gomaa]{R. S. Gomaa}
\address{${ }^{1}$Department of Mathematics, Faculty of Science, Mansoura University, 35516 Mansoura, Egypt}

\cortext[cor2]{Corresponding author}\ead{b\_desouky@yahoo.com {(B. S. Desouky)}}
\begin{abstract}
Recently, Komastu introduced the concept of poly-Cauchy numbers and polynomials which generalize Cauchy numbers and polynomials. In this paper, we introduce new generalization of poly-Cauchy and poly-Bernoulli numbers  and polynomials.  Also, we introduce new generalizations of Cauchy numbers and polynomials. Moreover, we derive some identities involving the new numbers and polynomials and some types of Stirling numbers. These gives generalization of some relations poly-Cauchy and poly-Bernoulli numbers and polynomials. Furthermore, we obtain some relations between the multiparameter poly-Cauchy numbers and polynomials and new multiparameter poly-Bernoulli numbers and polynomials.
\end{abstract}
\maketitle
\section{ Introduction}
 Comtet \cite{Comtet2} introduced two kinds of Cauchy numbers: The
first kind is given by
\begin{equation}
C_{n}=\int_{0}^{1}(x)_{n}dx,\text{ \ }n\in
\mathbb{Z}
_{\geq 0}
\end{equation}
and the second kind is given by
\begin{equation}
\hat{C}_{n}=\int_{0}^{1}(-x)_{n}dx,\text{ \ }n\in
\mathbb{Z}_{\geq 0},
\end{equation}
where $(x)_{n}=x(x-1)...(x-n+1).$
In \cite{Komastu1}, Komatsu introduced two kinds of poly-Cauchy numbers: The poly-Cauchy
numbers of the first kind $C_{n}^{(k)}$ as a generalization of the Cauchy numbers
are given by
\begin{equation}
C_{n}^{(k)}=\int_{0}^{1}...\int_{0}^{1}\left(
x_{1}x_{2}...x_{k}\right) _{n}dx_{1}dx_{2}...dx_{k},
\end{equation}
and the poly-Cauchy numbers of the second kind $\hat{C}_{n}^{(k)}$ are given by
\begin{equation}
\hat{C}_{n}^{(k)}=\int_{0}^{1}...\int_{0}^{1}\left(
-x_{1}x_{2}...x_{k}\right) _{n}dx_{1}dx_{2}...dx_{k}\text{ . }\left( n\in
\mathbb{Z}_{n\geq 0}, k\in
\mathbb{N}\right)
\end{equation}
The generating function of poly-Cauchy numbers \cite{Komastu1} is given by
\begin{equation*}
\text{Lif}_{k}\left( \ln \left( 1+x\right) \right) =\sum_{n=0}^{\infty
}C_{n}^{(k)}\frac{x^{n}}{n!},
\end{equation*}
where
\begin{equation*}
\text{Lif}_{k}(z)=\sum_{m=0}^{\infty }\frac{z^{m}}{m!(m+1)^{k}},
\end{equation*}
is the $k-th$ polylogarithm factorial function.\\
An explicit formula for $C_{n}^{(k)}$, see \cite{Komastu1} is given by
\begin{equation}
C_{n}^{(k)}=\sum_{m=0}^{n}\frac{s(n,m)}{(m+1)^{k}},\text{ }\left(
n\geq 0,\text{ }k\geq 1\right)
\end{equation}
where $s(n,m)$ are Stirling numbers of the first kind, see \cite{Gould}.\\
Komatsu \cite{Komastu3} introduced two kinds of poly-Cauchy numbers with a $q$ parameter:
The poly-Cauchy numbers with a $q$ parameter of the first kind $C_{n,q}^{(k)}$
are given by
\begin{equation}
C_{n,q}^{(k)}=\int_{0}^{1}...\int_{0}^{1}(
\prod_{i=0}^{n-1}(x_{1}x_{2}...x_{k}-iq)dx_{1}dx_{2}...dx_{k}\text{ , \ }\left(
n\geq 0,\text{ }k\geq 1\right)
\end{equation}
and the poly-Cauchy numbers with a $q$ parameter of the second kind $\hat{C}%
_{n,q}^{(k)}$ are given by
\begin{equation}
\hat{C}_{n,q}^{(k)}=\int_{0}^{1}...\int_{0}^{1}(
\prod_{i=0}^{n-1}(-x_{1}x_{2}...x_{k}-iq)dx_{1}dx_{2}...dx_{k}.\left(
n\geq 0,\text{ }k\geq 1\right)
\end{equation}
On the other hand, in 1997 Kaneko \cite{Kaneko} introduced the poly-Bernoulli numbers $%
B_{n}^{(k)}$ by
\begin{equation}
\frac{\text{Li}_{k}\left( 1-e^{-x}\right) }{1-e^{-x}}=\sum_{n=0}^{\infty
}B_{n}^{(k)}\frac{x^{n}}{n!},
\end{equation}
where
\begin{equation}
\text{Li}_{k}(z)=\sum_{m=1}^{\infty }\frac{z^{m}}{m^{k}},
\end{equation}
is the $k-th$ polylogarithm function.\\
An explicit formula for $B_{n}^{(k)}$, see \cite{Komastu2} is given by
\begin{equation}
B_{n}^{(k)}=(-1)^{n}\sum_{m=0}^{n}S(n,m)\frac{(-1)^{m}m!}{(m+1)^{k}}%
\text{ \ }\left( n\geq 0,\text{ }k\geq 1\right) ,
\end{equation}
where $S(n,m)$ are Stirling numbers of the second kind, see \cite{Gould}.\\
In Section 2, we present multiparameter poly-Cauchy numbers of the first kind and show that some results given in \cite{Komastu1,Komastu3,Komastu4} are special cases of
our result. In Section 3, we define multiparameter poly-Cauchy numbers of the second kind and obtain some relationships involving different types of
Stirling numbers. In Section 4, we define new generalization of Bernoulli numbers and derive some identites involving the new generalized poly-Cauchy numbers.
Finally, in Section 5, we define multiparameter poly-Cauchy polynomials and multiparameter poly-Bernoulli polynomials and derive some relationships between multiparameter poly-Cauchy polynomials and multiparameter poly-Bernoulli polynomials.
\section{Multiparameter poly-Cauchy numbers of the first kind}
\begin{definition}
Let $n\geq 0,$ $k\geq 1$ be integers, $\overline{\alpha }=\left( \alpha
_{0},\alpha _{1},...,\alpha _{n-1}\right) $ be a sequence of real numbers
and $L=$ $\left( \ell _{1},\ell _{2},...,\ell _{k}\right) $ be non-zero real
numbers. The multiparameter poly-Cauchy numbers of the first kind $C_{n,L}^{(k)}(\overline{\alpha} )$ are defined by
\begin{equation}\label{eq1}
 C_{n,L}^{(k)}(\overline{\alpha} )=\int_{0}^{\ell _{1}}\int_{0}^{\ell
_{2}}...\int_{0}^{\ell _{k}}\prod_{i=0}^{n-1}\left( x_{1}x_{2}...x_{k}-\alpha
_{i}\right)dx_{1}dx_{2}...dx_{k}.
\end{equation}
\end{definition}
We investigate some special cases:\\
\textbf{Case 1 }Setting $\alpha _{i}=i,$ $i=
0,1,...,n-1 ,$ $L=\left( 1,1,...,1\right) $ in \eqref{eq1}$,$ we have
\begin{equation}
C_{n,L}^{(k)}(\overline{i})=C_{n}^{(k)},  \quad \overline{i}=(0,1,...,n-1)
\end{equation}
where $C_{n}^{(k)}$ are poly-Cauchy numbers of the first kind, see \cite{Komastu3}.\\
\textbf{Case 2 \ }Setting $\alpha _{i}=iq,$ $i=
0,1,...,n-1 $ in \eqref{eq1}, we have
\begin{equation}
C_{n,L}^{(k)}(\overline{i}q)=C_{n,q,L}^{(k)}, \quad \overline{i}=(0,1,...,n-1)
\end{equation}
where $C_{n,q,L}^{(k)}$ are extension of poly-Cauchy numbers with a $q$
parameter, see \cite{Komastu3}.\\
\textbf{Case 3 }Setting $\alpha _{i}=iq,$ $i=0,1,...,n-1 ,$ $L=\left( 1,1,...,1\right) $ in \eqref{eq1}, we have
\begin{equation}
C_{n,L}^{(k)}(\overline{i}q)=C_{n,q}^{(k)}, \quad \overline{i}=(0,1,...,n-1)
\end{equation}
where $C_{n,q}^{(k)}$ ar ethe  poly-Cauchy numbers with a $q$ parameter, see \cite{Komastu3}.\\
If $k=1$ in \eqref{eq1}, we define the generalized Cauchy
numbers of the first kind associated with $\overline{\alpha }=\left( \alpha _{0},\alpha
_{1},...,\alpha _{n-1}\right) $, called multiparameter Cauchy numbers of the first kind, by
\begin{equation}\label{eq80}
C_{n,\overline{\alpha }}=\int_{0}^{\ell}\left( x-\alpha _{0}\right)
\left( x-\alpha _{1}\right) ...\left( x-\alpha _{n-1}\right) dx.
\end{equation}
\textbf{Case 4 }Setting $\alpha _{i}=i,$ $i=0,1,...,n-1 $ in \eqref{eq80}, we have
\begin{equation}
C_{n,\overline{i}}=C_{n}, \quad \overline{i}=(
0,1,...,n-1)
\end{equation}
where $C_{n}$ are Cauchy numbers of the first kind, see \cite{Merlini}.\\
\textbf{Case 5 }Setting $\alpha _{i}=iq,$ $i=0,1,...,n-1 $ in \eqref{eq80}, we obtain
\begin{equation}
C_{n,\overline{i}q}=C_{n,q}, \quad \overline{i}=(
0,1,...,n-1)
\end{equation}
where $C_{n,q}$ are Cauchy numbers of the first kind with a parameter $q$, see
\cite{Komastu3}.\\
Multiparameter poly-Cauchy numbers of the first kind $C_{n,L}^{(k)}(\overline{\alpha })$ can be expressed in terms of different types of the Stirling numbers
as follows:
\begin{theorem}
For a sequence of real numbers $\overline{\alpha }=\left( \alpha
_{0},\alpha _{1},...,\alpha _{n-1}\right) ,$
\begin{equation}\label{eq2}
C_{n,L}^{(k)}(\overline{\alpha })=\sum_{m=0}^{n}\frac{s_{\overline%
{\alpha }}\left( n,m\right) }{(m+1)^{k}}\left( \ell _{1}\ell _{2}...\ell
_{k}\right) ^{m+1},
\end{equation}
where $s_{\overline{\alpha }}\left( n,m\right) $ are the generalized Stirling
numbers of the first kind, called Comtet numbers of the first kind, see \cite{Comtet1}, are defined as
\begin{equation}\label{eq6000}
(x;\overline{\alpha})_{m}=\sum_{i=0}^{m}s_{\overline{\alpha }}( m,i)x^{i},
\end{equation}
where $(x;\overline{\alpha})_{m}=\prod_{i=0}^{m-1}(x-\alpha_{i})$.
\end{theorem}
\begin{proof}
Using equation \eqref{eq1} and \eqref{eq6000}, hence\\
$$
C_{n,L}^{(k)}(\overline{\alpha })=\int_{0}^{\ell_{1}}\int_{0}^{\ell_{2}}...\int_{0}^{\ell_{k}}\sum_{m=0}^{n}s_{\overline{\alpha }}( n,m)(x_{1}x_{2}...x_{k})^{m}dx_{1}dx_{2}...dx_{k},$$ then
 we obtain \eqref{eq2}.
\end{proof}
\begin{corollary}
If $k=1$ in \eqref{eq2}, we have the following relationship
\begin{equation}
C_{n,\overline{\alpha} ,\ell }=\sum_{m=0}^{n}\frac{s_{\overline{%
\alpha }}\left( n,m\right) }{m+1}\ell^{m+1},
\end{equation}
between generalized Cauchy numbers of the first kind and generalized Stirling
numbers of first kind.
\end{corollary}
\begin{theorem}
For $\overline{\alpha }=\left( \alpha _{0},\alpha _{1},...,\alpha
_{n-1}\right) ,$ we have
\begin{equation}\label{eq3}
C_{n,L}^{(k)}(\overline{\alpha })=\sum_{j=0}^{n}\sum%
_{m=j}^{n}\frac{S\left( n,m;\overline{\alpha }\right) s(m,j)}{%
(j+1)^{k}}\left( \ell _{1}\ell _{2}...\ell _{k}\right) ^{j+1},
\end{equation}
which gives a relationship of multiparameter poly-Cauchy numbers of the first kind in terms of the multiparameter non-central Stirling numbers of the second kind and Stirling numbers of the first kind, see \cite{Cakic1,Desouky}.
\end{theorem}
\begin{proof}
Using equation \eqref{eq1} and from the definition of multiparameter non-central Stirling numbers
of the second kind, we obtain \\
$C_{n,L}^{(k)}(\overline{\alpha })=\int_{0}^{\ell
_{1}}\int_{0}^{\ell _{2}}...\int_{0}^{\ell
_{k}}\sum_{m=0}^{n}S\left( n,m;\overline{\alpha }\right) \left(
x_{1}x_{2}...x_{k}\right) _{m}dx_{1}dx_{2}...dx_{k},$\\
and from the definition of Stirling numbers of the first kind, we easily obtain \eqref{eq3}.
\end{proof}
\begin{corollary}
If $k=1$ in \eqref{eq3}, then the generalized Cauchy numbers can be expressed in terms of the  multiparameter non-central Stirling numbers of the second kind and Stirling numbers of the first kind as
\begin{equation}
C_{n,\overline{\alpha },\ell }=\sum_{j=0}^{n}\sum_{m=j}^{n}%
\frac{S\left( n,m;\overline{\alpha }\right) s(m,j)}{j+1}\ell^{j+1}.
\end{equation}
\end{corollary}
\begin{theorem}
For $\ \overline{\alpha }=\left( \alpha _{0},\alpha _{1},...,\alpha
_{n-1}\right),$ we have
\begin{equation}
C_{n,L}^{(k)}(\overline{\alpha })=\sum_{m=0}^{n}S(n,m;\overline{%
\alpha })C_{n}^{(k)},
\end{equation}
where $S(n,m;\overline{\alpha })$ are the multiparameter non-central Stirling numbers of
the second kind and $C_{n}^{(k)}$ are poly-Cauchy numbers of the first kind.
\end{theorem}
\begin{theorem}
An explicit formula of $\ C_{n,L}^{(k)}(\overline{\alpha })$ can be
expressed as
\begin{equation}\label{eq5}
C_{n,L}^{(k)}(\overline{\alpha })=(-1)^{n} \prod_{i=0}^{n-1}\alpha_{i} \sum_{m=0}^{n}\frac{P_{m}\left( -H_{n,%
\overline{\alpha }}^{(1)},-H_{n,\overline{\alpha }}^{(2)},...,-H_{n,%
\overline{\alpha }}^{(m)}\right) }{(m+1)^{k}}\left( \ell _{1}\ell
_{2}...\ell _{k}\right) ^{m+1},
\end{equation}
where $P_{m}\left( x_{1},x_{2},...,x_{m}\right)
=\sum_{k_{1}+2k_{2}+3k_{3}...=m}\frac{1}{k_{1}!k_{2}!...}\left( \frac{x_{1}}{1}\right) ^{k_{1}}\left( \frac{x_{2}}{2}\right) ^{k_{2}}...$ is the modified Bell polynomial, see \cite[(p.308, Definition 2)]{Candelpergher}
and $H_{n,\overline{\alpha }}%
^{(k)}=\sum_{j=0}^{n-1}\frac{1}{(\alpha _{j})^{k}}$ are the generalized
harmonic numbers, see \cite{Cakic}.
\end{theorem}
\begin{proof}
From \eqref{eq1}\\
$C_{n,L}^{(k)}(\overline{\alpha })=\int_{0}^{\ell
_{1}}\int_{0}^{\ell _{2}}...\int_{0}^{\ell
_{k}}(-1)^{n}\left( \alpha _{0}\alpha _{1}...\alpha _{n-1}\right)\prod_{i=0}^{n-1}(1-\frac{x_{1}x_{2}...x_{k}}{\alpha_{i}})dx_{1}dx_{2}...dx_{k}$

\ \ \ \ \ \ \ \ \ \ =$\int_{0}^{\ell
_{1}}\int_{0}^{\ell _{2}}...\int_{0}^{\ell
_{k}}(-1)^{n}\left( \alpha _{0}\alpha _{1}...\alpha _{n-1}\right)
e^{\sum_{i=0}^{n-1}\ln \left( 1-\frac{x_{1}x_{2}...x_{k}}{\alpha _{i}%
}\right) }dx_{1}dx_{2}...dx_{k}$

\ \ \ \ \ \ \ \ \ \ $=\int_{0}^{\ell _{1}}\int_{0}^{\ell
_{2}}...\int_{0}^{\ell _{k}}(-1)^{n} \prod_{i=0}^{n-1}\alpha_{i} e^{-\sum_{j=1}^{\infty
}\sum_{i=0}^{n-1}\frac{1}{(\alpha _{i})^{j}}.\frac{\left(
x_{1}x_{2}...x_{k}\right) ^{j}}{j}}dx_{1}dx_{2}...dx_{k}$

\ \ \ \ \ \ \ \ \ \ $=\int_{0}^{\ell _{1}}\int_{0}^{\ell
_{2}}...\int_{0}^{\ell _{k}}(-1)^{n} \prod_{i=0}^{n-1}\alpha_{i} e^{-\sum_{j=1}^{\infty }\frac{H_{n,%
\overline{\alpha }}^{(j)}}{j}\left( x_{1}x_{2}...x_{k}\right) ^{j}}dx_{1}dx_{2}...dx_{k},$\\
then we obtain \eqref{eq5}.
\end{proof}
\section{Multiparameter poly-Cauchy numbers of the second kind}
\begin{definition}
Let $n\geq 0,$ $k\geq 1$ be integers, $\overline{\alpha }=\left( \alpha
_{0},\alpha _{1},...,\alpha _{n-1}\right) $ be a sequence of real numbers
and $L=$ $\left( \ell _{1},\ell _{2},...,\ell _{k}\right) $ be non-zero real
numbers.  The multiparameter poly-Cauchy numbers of the second kind $\hat{C%
}_{n,L}^{(k)}\left(\overline{\alpha }\right)$ are defined by
\begin{equation}\label{eq6}
\hat{C}_{n,L}^{(k)}(\overline{\alpha })=\int_{0}^{\ell _{1}}\int_{0}^{\ell
_{2}}...\int_{0}^{\ell _{k}}\prod_{i=0}^{n-1}\left( -x_{1}x_{2}...x_{k}-\alpha
_{i}\right)dx_{1}dx_{2}...dx_{k}.
\end{equation}
\end{definition}
We investigate some special cases:\\
\textbf{Case 1 } Setting $\alpha _{i}=i,$ $i=(0,1,...,n-1),
L=\left( 1,1,...,1\right) $ in \eqref{eq6}, we have
\begin{equation}
\hat{C}_{n,\overline{i},L}^{(k)}=\hat{C}_{n}^{(k)}, \quad \overline{i}=( 0,1,...,n-1),
\end{equation}
where $\hat{C}_{n}^{(k)}$ are poly-Cauchy numbers of the the second kind, see \cite{Komastu1}.\\
\textbf{Case 2 } Setting $\alpha _{i}=iq,$ $i=(0,1,...,n-1) $ in
\eqref{eq6}, we have
\begin{equation}
\hat{C}_{n,\overline{i}q,L}^{(k)}=\hat{C}_{n,q,L}^{(k)}, \quad \overline{i}=( 0,1,...,n-1),
\end{equation}
where $\hat{C}_{n,q,L}^{(k)}$ are extension of poly-Cauchy numbers with a $q$
parameter, see \cite{Komastu4}.\\
\textbf{Case 3 } Setting $\alpha _{i}=iq,$ $i=(0,1,...,n-1) ,$ $L=\left( 1,1,...,1\right) $ in \eqref{eq6}, we have
\begin{equation}
\hat{C}_{n,\overline{i}q,L}^{(k)}=\hat{C}_{n,q}^{(k)}, \quad \overline{i}=(0,1,...,n-1),
\end{equation}
where $\hat{C}_{n,q}^{(k)}$ are poly-Cauchy numbers with a $q$ parameter,
see \cite{Komastu3}.\\
Also, we define the generalized Cauchy numbers of the second kind associated with $\overline{%
\alpha }=\left( \alpha _{0},\alpha _{1},...,\alpha _{n-1}\right) $, called multiparameter Cauchy numbers of the second kind, by
\begin{equation}\label{eq7}
\hat{C}_{n,\overline{\alpha }}=\int_{0}^{\ell}\left( -x-\alpha
_{0}\right) \left( -x-\alpha _{1}\right) ...\left( -x-\alpha _{n-1}\right)
dx.
\end{equation}
\textbf{Case 4 } Setting $\alpha _{i}=i,$ $i=0,1,...,n-1 $ in \eqref{eq7}, we have
\begin{equation}
\hat{C}_{n,\overline{i}}=\hat{C}_{n}, \quad \overline{i}=\left(
0,1,...,n-1\right)
\end{equation}
where $\hat{C}_{n}$ are Cauchy numbers of the second kind, see \cite{Merlini}.\\
\textbf{Case 5 } Setting $\alpha _{i}=iq,$ $i=0,1,...,n-1 $ in \eqref{eq7},
we obtain
\begin{equation}
\hat{C}_{n,\overline{i}q}=\hat{C}_{n,q}, \quad \overline{i}=( 0,1,...,n-1)
\end{equation}
where $\hat{C}_{n,q}$ are Cauchy numbers of the second kind with a parameter $q$%
, see \cite{Komastu3}.
\begin{theorem}
$ \hat{C}_{n,L}^{(k)}(\overline{\alpha })$ can be expressed in terms of the signless generalized Stirling numbers of the first kind as
\begin{equation}\label{eq8}
\hat{C}_{n,L}^{(k)}(\overline{\alpha })=\sum_{m=0}^{n}\frac{(-1)^{n}|s_{%
\overline{\alpha }}(n,m)|\left( \ell _{1}\ell _{2}...\ell _{k}\right)
^{m+1}}{\left( m+1\right) ^{k}},
\end{equation}
where $|s_{\overline{\alpha }}(n,m)|$ are the signless generalized Stirling
numbers of first kind, see \cite{Comtet1}.
\end{theorem}
\begin{proof}
Using equation \eqref{eq6}, from the definition of the signless generalized Stirling numbers
of the first kind, we obtain \eqref{eq8}.
\end{proof}
\begin{corollary}
If $k=1$ in Theorem 3.1, we have
\begin{equation}
\hat{C}_{n,\overline{\alpha }}=\sum_{m=0}^{n}\frac{(-1)^{n}|s_{%
\overline{\alpha }}(n,m)|\left( \ell \right) ^{m+1}}{m+1},
\end{equation}
which gives the generalized Cauchy numbers of the second kind in terms of the signless generalized Stirling numbers of the first kind.
\end{corollary}
\begin{theorem}
For \ $\overline{\alpha }=\left( \alpha _{0},\alpha _{1},...,\alpha
_{n-1}\right) ,$ then the multiparameter poly-Cauchy numbers of the second kind can be expressed in terms of the multiparameter non-central Stirling numbers of the first kind, Lah numbers, see \cite[p. 5]{Petkovsek} and poly-Cauchy numbers of the first kind as follows
\begin{equation}\label{eq9}
\hat{C}_{n,L}(\overline{\alpha })^{(k)}=\sum_{\ell
=0}^{n}\sum_{m=\ell }^{n}s\left( n,m;\overline{\alpha }\right)
L(m,\ell)C_{\ell }^{(k)}.
\end{equation}
\end{theorem}
\begin{proof}
From \eqref{eq6} and the definition of the multiparameter non-central Stirling numbers of the first kind, we
have\\
\begin{equation*}
\hat{C}_{n,L}^{(k)}(\overline{\alpha })=\int_{0}^{\ell
_{1}}\int_{0}^{\ell _{2}}...\int_{0}^{\ell
_{k}}\sum_{m=0}^{n}s\left( n,m;\overline{\alpha }\right) \left(
-x_{1}x_{2}...x_{k}\right) _{m}dx_{1}dx_{2}...dx_{k},
\end{equation*}
from definition of Lah numbers
\begin{equation*}
(-x_{1}x_{2}...x_{k})_{m}=\sum_{\ell=0}^{m}L(m,\ell)(x_{1}x_{2}...x_{k})_{\ell},
\end{equation*}
then
\begin{equation*}
\hat{C}_{n,L}^{(k)}(\overline{\alpha })=\int_{0}^{\ell
_{1}}\int_{0}^{\ell _{2}}...\int_{0}^{\ell
_{k}}\sum_{m=0}^{n}s\left( n,m;\overline{\alpha }\right) \sum_{\ell=0}^{m}L(m,\ell)(x_{1}x_{2}...x_{k})_{\ell}dx_{1}dx_{2}...dx_{k},
\end{equation*}
 and from the definition of poly-Cauchy numbers of
the first kind, yields \eqref{eq9}.
\end{proof}
\begin{corollary}
If $k=1$ in \eqref{eq9}, we have the following relationship
\begin{equation}
\hat{C}_{n,\ell _{1}}(\overline{\alpha })=\sum_{\ell
=0}^{n}\sum_{m=\ell }^{n}s\left( n,m;\overline{\alpha }\right)
L(m,\ell)C_{\ell _{1}},
\end{equation}
between the generalized Cauchy numbers of the second kind, the multiparameter non-central Stirling
numbers of the first kind, Lah numbers and Cauchy numbers of the first kind.
\end{corollary}
\section{Multiparameter poly-Bernoulli numbers}
We define the multiparameter poly-Bernoulli numbers $B_{n,\overline{\alpha} ,L}^{(k)}$ in terms of the generalized Stirling numbers of the second kind as
\begin{equation}\label{eq10}
B_{n,\overline{\alpha},L}^{(k)}=\sum_{m=0}^{n}(-1)^{n-m}m!\frac{S_{\overline{\alpha }}(n,m)m!}{(m+1)^{k}}\left( \ell _{1}\ell _{2}...\ell _{k}\right)
^{m+1},
\end{equation}
where $S_{\overline{\alpha }}(n,m)$ are the generalized Stirling numbers of
the second kind, see \cite{Comtet2}.
\begin{theorem}
The generating fuction of \ $B_{n,\overline{\alpha} ,L}^{(k)}$ is given by
\begin{equation}\label{eq11}
\sum_{n=0}^{\infty }B_{n,\overline{\alpha} ,L}^{(k)}\frac{t^{n}}{n!}%
=\sum_{j=0}^{\infty }\sum_{m=j}^{n}(-1)^{m}m!\frac{%
e^{-t\alpha _{j}}}{(m+1)^{k}(\alpha _{j})_{m}}\left( \ell _{1}\ell
_{2}...\ell _{k}\right) ^{m+1},
\end{equation}
where $(\alpha _{j})_{m}=\prod_{\substack{ i=0  \\ i\neq j}}%
^{m}\left( \alpha _{j}-\alpha _{i}\right) .$
\end{theorem}
\begin{proof}
From equation \eqref{eq10}, then
\begin{eqnarray*}
\sum_{n=0}^{\infty }B_{n,\overline{\alpha} ,L}^{(k)}\frac{t^{n}}{n!}
&=&\sum_{n=0}^{\infty }\sum_{m=0}^{n}(-1)^{n-m}\frac{ S_{%
\overline{\alpha }}(n,m)m!}{(m+1)^{k}}\left( \ell _{1}\ell _{2}...\ell
_{k}\right) ^{m+1}\frac{t^{n}}{n!}\\
&=&\sum_{m=0}^{\infty}(-1)^{m}\frac{m!\left( \ell _{1}\ell _{2}...\ell _{k}\right) ^{m+1}}{%
(m+1)^{k}}\sum_{n=m}^{\infty }S_{\overline{\alpha }}(n,m)\frac{%
(-t)^{n}}{n!},
\end{eqnarray*}
and from the generating function of the generalized Stirling numbers of the second
kind, see (\cite{Comtet2}, Eq.(9)), yields \eqref{eq11}.\\
\end{proof}
In addition, there are some relationships between $\hat{C}_{n,L}^{(k)}(\overline{\alpha})$ and
$B_{n,\overline{\alpha} ,L}^{(k)}$ 
\begin{theorem}
For $n\geq 0,$ we have
\begin{equation}\label{eq12}
\hat{C}_{n,L}^{(k)}(\overline{\alpha})=\sum_{j=0}^{n}\sum_{m=0}^{n}(-1)^{n}
\frac{s_{\overline{\alpha }}(m,j)|s_{\overline{\alpha }}(n,m)|}{m!}B_{j,\overline{\alpha} ,L}^{(k)},
\end{equation}
\begin{equation}\label{eq13}
B_{n,\overline{\alpha} ,L}^{(k)} =\sum_{j=0}^{n}\sum_{m=0}^{n}(-1)^{n-m}%
\frac{S_{\overline{\alpha }}(m,j)S_{\overline{\alpha }}(n,m)}{m!}\hat{C}%
_{j,L}^{(k)}(\overline{\alpha}).
\end{equation}
\end{theorem}
\begin{proof}
For the first identiy, we have
\begin{eqnarray*}
RHS&=&\sum_{j=0}^{n}\sum_{m=0}^{n}(-1)^{n}\frac{s_{\overline{\alpha }}(m,j)|s_{\overline{\alpha }}(n,m)|}{m!}B_{j,\overline{\alpha} ,L}^{(k)}\\
&=&\sum_{m=0}^{n}(-1)^{n}\frac{|s_{\overline{\alpha }}(n,m)|}{m!}\sum_{j=0}^{n}s_{\overline{\alpha }}(m,j)\sum_{i=0}^{j}(-1)^{j-i}i!\frac{S_{\overline{\alpha }}(j,i)}{(
i+1) ^{k}}\left( \ell _{1}\ell _{2}...\ell _{k}\right) ^{i+1}\\
&=&\sum_{m=0}^{n}\frac{|s_{\overline{\alpha }}(n,m)|}{m!}\sum_{i=0}^{n}i!\frac{\left( \ell _{1}\ell _{2}...\ell_{k}\right) ^{i+1}}{\left( i+1\right) ^{k}}\sum_{j=i}^{n}(-1)^{j-i}s_{\overline{\alpha }}(m,j)S_{\overline{\alpha }}(j,i),
\end{eqnarray*}
from
\begin{equation*}
\sum_{j=i}^{n}(-1)^{j-i}s_{\overline{\alpha }}(m,j)S_{\overline{\alpha }}(j,i) =
\begin{cases}
1 & \text{if } i=m\\
0 & \text{if } i\neq m,
\end{cases}
\end{equation*}
 then we obtain \eqref{eq12}.\\
 Similarly, we can prove \eqref{eq13}.\\
 \end{proof}
If we put $k=1$ in Theorem 4.2, we have the following Colorllary:
\begin{corollary}
For $n\geq 0,$ we have
\begin{equation}
\hat{C}_{n,\overline{\alpha }}=\sum_{j=0}^{n}\sum_{m=0}^{n}\frac{s_{%
\overline{\alpha }}(m,j)|s_{\overline{\alpha }}(n,m)|}{m!}B_{j,\overline{\alpha }},
\end{equation}
\begin{equation}
B_{n,\overline{\alpha }}=\sum_{j=0}^{n}\sum_{m=0}^{n}(-1)^{n-m}\frac{S_{%
\overline{\alpha }}(m,j)S_{\overline{\alpha }}(n,m)}{m!}\hat{C}_{j,%
\overline{\alpha }}.
\end{equation}
\end{corollary}
\begin{theorem}
For $n\geq 0,$ we have
\begin{equation}\label{eq14}
C_{n,L}^{(k)}(\overline{\alpha})=\sum_{j=0}^{n}\sum_{m=0}^{n}\frac{s_{%
\overline{\alpha }}(m,j)s_{\overline{\alpha }}(n,m)}{m!}B_{j,\overline{%
\alpha} ,L}^{(k)},
\end{equation}
\begin{equation}\label{eq15}
B_{n,\overline{\alpha} ,L}^{(k)}=\sum_{j=0}^{n}\sum_{m=0}^{n}(-1)^{n-m}%
\frac{S_{\overline{\alpha }}(m,j)S_{\overline{\alpha }}(n,m)}{m!}C_{j,L}^{(k)}(\overline{\alpha}).
\end{equation}
\end{theorem}
Setting $k=1$ in Theorem 4.3 , we otain the following Colorally
\begin{corollary}
For $n\geq 0,$ we have
\begin{equation}
C_{n,\overline{\alpha }}=\sum_{j=0}^{n}\sum_{m=0}^{n}\frac{s_{\overline{%
\alpha }}(m,j)s_{\overline{\alpha }}(n,m)}{m!}B_{j,\overline{\alpha }},
\end{equation}
\begin{equation}
B_{n,\overline{\alpha }}=\sum_{j=0}^{n}\sum_{m=0}^{n}(-1)^{n-m}\frac{S_{%
\overline{\alpha }}(m,j)S_{\overline{\alpha }}(n,m)}{m!}C_{j,\overline{%
\alpha }}.
\end{equation}
\end{corollary}
\section{Multiparameter poly-Cauchy and multiparameter
poly-Bernoulli polynomials}
\begin{definition}
Multiparameter poly-Cauchy polynomials of the first and second kind,
respectively, are defined by
\begin{equation}\label{eq15}
C_{n,L}^{(k)}(z;\overline{\alpha })=\int_{0}^{\ell
_{1}}\int_{0}^{\ell _{2}}...\int_{0}^{\ell _{k}}\prod_{i=0}^{n-1}\left(
x_{1}x_{2}...x_{k}-\alpha _{i}-z\right)
dx_{1}dx_{2}...dx_{k},
\end{equation}
\begin{equation}\label{eq16}
\hat{C}_{n,L}^{(k)}(z;\overline{\alpha })=\int_{0}^{\ell
_{1}}\int_{0}^{\ell _{2}}...\int_{0}^{\ell _{k}}\prod_{i=0}^{n-1}\left(
-x_{1}x_{2}...x_{k}-\alpha _{i}+z\right)dx_{1}dx_{2}...dx_{k}.
\end{equation}
\end{definition}
Setting $k=1$ in \eqref{eq15} and \eqref{eq16}, we can define the generalized Cauchy polynomials of the first and second kind as follows
\begin{definition}
Generalized Cauchy polynomials of the first and the second kind, respectively, are
defined by
\begin{equation}
C_{n,\overline{\alpha }}(z)=\int_{0}^{\ell }\left( x-\alpha
_{0}-z\right) \left( x-\alpha _{1}-z\right) ...\left( x-\alpha
_{n-1}-z\right) dx,
\end{equation}
\begin{equation}
\hat{C}_{n,\overline{\alpha }}(z)=\int_{0}^{\ell }\left(
-x-\alpha _{0}+z\right) \left( -x-\alpha _{1}+z\right) ...\left( -x-\alpha
_{n-1}+z\right) dx.
\end{equation}
\end{definition}
\begin{theorem}
i) $C_{n,\overline{\alpha },L}^{(k)}(z)$ are expressed in terms of
the generalized Stirling numbers of the first kind as
\begin{equation}\label{eq20}
 C_{n,L}^{(k)}(z;\overline{\alpha })=\sum_{i=0}^{n}\sum_{m=i}^{n}(-1)^{i}\binom{m}{i}  \frac{s_{\overline{\alpha }}\left( n,m\right)
\left( \ell _{1}\ell _{2}...\ell _{k}\right) ^{m-i+1}}{\left( m-i+1\right)
^{k}}(z)^{i},
\end{equation}
ii) $\hat{C}_{n,L}^{(k)}(z;\overline{\alpha })$ are expressed in terms of
the signless generalized Stirling numbers of the first kind as
\begin{equation}\label{eq21}
\hat{C}_{n,L}^{(k)}(z;\overline{\alpha })=\sum_{i=0}^{n} \sum_{m=i}^{n}(-1)^{i+n}\binom{m}{i}\frac{|s_{\overline{\alpha }}\left( n,m\right)
|\left( \ell _{1}\ell _{2}...\ell _{k}\right) ^{m-i+1}}{\left( m-i+1\right)
^{k}}(z)^{i}.
\end{equation}
\end{theorem}
\begin{proof}
For the first identity, from \eqref{eq15} and the definition of the generalized Stirling
numbers of the first kind, we obtain\\
$C_{n,L}^{(k)}(z;\overline{\alpha })=\int_{0}^{\ell
_{1}}\int_{0}^{\ell _{2}}...\int_{0}^{\ell
_{k}}\sum_{m=0}^{n}s_{\overline{\alpha }}\left( n,m\right) \left(
x_{1}x_{2}...x_{k}-z\right)
^{m}dx_{1}dx_{2}...dx_{k}$

\ \ \ \ \ \ \ \ \ = $\int_{0}^{\ell
_{1}}\int_{0}^{\ell _{2}}...\int_{0}^{\ell
_{k}}\sum_{m=0}^{n}s_{\overline{\alpha }}\left( n,m\right)
\sum_{i=0}^{m}\binom{m}{i} \left( -z\right) ^{i}\left(
x_{1}x_{2}...x_{k}\right) ^{m-i}dx_{1}dx_{2}...dx_{k}$

\ \ \ \ \ \ \  =$\sum_{i=0}^{n}\sum_{m=i}^{n}\binom{m}{i} \left( -z\right) ^{i}\frac{s_{%
\overline{\alpha }}\left( n,m\right) \left( \ell _{1}\ell _{2}...\ell
_{k}\right) ^{m-i+1}}{\left( m-i+1\right) ^{k}}.$\\
For the second identity, from \eqref{eq16} and the definition of the signless generalized
Stirling numbers of the first kind, we obtain
{\small \begin{eqnarray*}
\hat{C}_{n,\overline{\alpha },L}^{(k)}(z)&=&\int_{0}^{\ell
_{1}}\int_{0}^{\ell _{2}}...\int_{0}^{\ell
_{k}}\sum_{m=0}^{n}(-1)^{n}|s_{\overline{\alpha }}\left(
n,m\right) |\left( x_{1}x_{2}...x_{k}-z\right)
^{m}dx_{1}dx_{2}...dx_{k}\\
&=&(-1)^{n}\int_{0}^{\ell
_{1}}\int_{0}^{\ell _{2}}...\int_{0}^{\ell
_{k}}\sum_{m=0}^{n}|s_{\overline{\alpha }}\left( n,m\right)
|\sum_{i=0}^{m}\binom{m}{i} \left( -z\right) ^{i}\left(
x_{1}x_{2}...x_{k}\right) ^{m-i}dx_{1}dx_{2}...dx_{k}\\
 &=&(-1)^{n}\sum_{i=0}^{n}\sum_{m=i}^{n}\binom{m}{i} \left( -z\right) ^{i}\frac{|s_{%
\overline{\alpha }}\left( n,m\right) |\left( \ell _{1}\ell _{2}...\ell
_{k}\right) ^{m-i+1}}{\left( m-i+1\right) ^{k}}.
\end{eqnarray*}}
\end{proof}
Setting $k=1$ in \eqref{eq20} and \eqref{eq21}, we obtain the following Corollary.
\begin{corollary}
Generalized Cauchy polynomials of the first kind $C_{n,\overline{\alpha }}(z)$
are expressed in terms of the generalized Stirling numbers of the first kind as
\begin{equation}
C_{n,\overline{\alpha }}(z)=\sum_{i=0}^{n}\sum_{m=i}^{n}(-1)^{i}\binom{m}{i}\frac{s_{\overline{\alpha }}\left( n,m\right)
\left( \ell \right) ^{m-i+1}}{m-i+1}z^{i},
\end{equation}
generalized Cauchy polynomials of first kind $\hat{C}_{n,\overline{%
\alpha }}(z)$ are expressed in terms of the signless generalized Stirling
numbers of the first kind as
\begin{equation}
\hat{C}_{n,\overline{\alpha }}(z)=\sum_{i=0}^{n}\sum_{m=i}^{n}(-1)^{i+n}\binom{m}{i}\frac{|s_{\overset{-}{\alpha }}\left( n,m\right)
|\left( \ell \right) ^{m-i+1}}{m-i+1}z^{i}.
\end{equation}
\end{corollary}
Coppo and Candelpergher \cite{Coppo} and Bayad and Hamahata \cite{Bayad} introduced
the poly-Bernoulli polynomial $B_{n}^{(k)}(z)$ by
\begin{equation*}
\frac{\text{Li}_{k}\left( 1-e^{-x}\right) }{1-e^{-x}}e^{-xz}=\sum_{n=0}^{%
\infty }B_{n}^{(k)}(z)\frac{t^{n}}{n!}
\end{equation*}
and
\begin{equation*}
\frac{\text{Li}_{k}\left( 1-e^{-x}\right) }{1-e^{-x}}e^{xz}=\sum_{n=0}^{%
\infty }B_{n}^{(k)}(z)\frac{t^{n}}{n!}.
\end{equation*}
Also, Komastu \cite{Kamano} introduced poly-Bernoulli polynomials $B_{n}^{(k)}(z)$ by
\begin{equation*}
B_{n}^{(k)}(z)=(-1)^{n}\sum_{m=0}^{n}S(n,m)(-1)^{m}m!\sum_{i=0}^{m}\binom{m}{i} \frac{(-z)^{i}}{\left( m-i+1\right)
^{k}}.
\end{equation*}
Next, we introduce the multiparameter Bernoulli polynomials in terms of the generalized
Stirling numbers of the second kind as
\begin{equation}\label{eq1000}
B_{n,\overline{\alpha} ,L}^{(k)}(z)=(-1)^{n}\sum_{i=0}^{n}\sum_{m=i}^{n}(-1)^{m}m!\binom{m}{i} \frac{S_{\overline{%
\alpha }}(n,m)\left( \ell _{1}\ell _{2}...\ell _{k}\right) ^{m-i+1}}{\left(
m-i+1\right) ^{k}}(-z)^{i}.
\end{equation}
From the last equation and the definition of the generating function of generalized
Stirling numbers of the second kind, see \cite{Comtet2}, the generating function of $%
B_{n,\overline{\alpha} ,L}^{(k)}(z)$ is defined as
\begin{equation}
\sum_{n=0}^{\infty }B_{n,\overline{\alpha} ,L}^{(k)}(z)\frac{t^{n}%
}{n!}=\sum_{i=0}^{n}\sum_{m=i}^{n}\text{ }(-1)^{m}\binom{m}{i} \frac{\left( \ell _{1}\ell _{2}...\ell _{k}\right)
^{m-i+1}(-z)^{i}}{\left( m-i+1\right) ^{k}}\sum_{j=0}^{\infty }\frac{%
e^{-t\alpha _{j}}}{\prod_{\substack{ i=0 \\ i\neq j}}\left( \alpha
_{j}-\alpha _{i}\right) }.
\end{equation}
\begin{theorem}
For $n\geq 0,$ we have
\begin{equation}\label{eq909}
B_{n,\overline{\alpha} ,L}^{(k)}(z)
=\sum_{j=0}^{n}\sum_{m=0}^{n}(-1)^{n-m} S_{\overline{\alpha }}
(m,j)S_{\overline{\alpha }}(n,m)m!C_{j,L}^{k}(z;\overline{\alpha }),
\end{equation}
\begin{equation}\label{eq901}
B_{n,\overline{\alpha} ,L}^{(k)}(z)=\sum_{j=0}^{n}\sum_{m=0}^{n}(-1)^{n-m} S_{\overline{\alpha }}(m,j)S_{\overline{\alpha }}(n,m)m!\hat{C}_{j,L}^{(k)}(z;\overline{\alpha }).
\end{equation}
and
\begin{equation}\label{eq903}
C_{n,L}^{(k)}(z;\overline{\alpha })=\sum_{j=0}^{n}\sum_{m=0}^{n}\frac{s_{%
\overline{\alpha }}(m,j)s_{\overline{\alpha }}(n,m)}{m!}B_{j,\overline{
\alpha} ,L}^{(k)}(z),
\end{equation}
\begin{equation}\label{eq904}
\hat{C}_{n,L}^{(k)}(z;\overline{\alpha })=\sum_{j=0}^{n}\sum_{m=0}^{n}%
(-1)^{n}\frac{s_{\overline{\alpha }}(m,j)|s_{\overline{\alpha }}(n,m)|}{m!}B_{j,%
\overline{\alpha} ,L}^{(k)}(z).
\end{equation}
\end{theorem}
\begin{proof} We prove \eqref{eq909} as follows. From \eqref{eq20}, we have
{\small \begin{eqnarray*}
RHS \text{ }of \text{ } \eqref{eq909}&=&\sum_{j=0}^{n}\sum_{m=0}^{n}(-1)^{n-m} S_{\overline{\alpha }}
(m,j)S_{\overline{\alpha }}(n,m)m!C_{j,L}^{k}(z;\overline{\alpha }) \\
&=&\sum_{j=0}^{n}\sum_{m=0}^{n}(-1)^{n-m} S_{\overline{\alpha }%
}(m,j)S_{\overline{\alpha} }(n,m)m!\sum_{i=0}^{j}\sum_{l=i}^{j}\binom{l}{i}s_{\overline{\alpha} }(j,l)\frac{(\ell_{1}\ell_{2}...\ell_{k})^{l-i+1}}{(l-i+1)^{k}}(-z)^{i}\\
&=&\sum_{m=0}^{n}(-1)^{n-m}S_{\overline{\alpha }}(n,m)m!\sum_{l=0}^{n}\sum_{i=0}^{l}\binom{l}{i} \frac{(\ell_{1}\ell_{2}...\ell_{k})^{l-i+1}}{(l-i+1)^{k}}(-z)^{i}\sum_{j=l}^{n}S_{\overline{\alpha }}(m,j)s_{\overline{\alpha} }(j,l),
\end{eqnarray*}}
since
\begin{equation*}
\sum_{j=l}^{n}s_{\overline{\alpha }}(j,l)S_{%
\overline{\alpha }}(m,j) = 
\begin{cases}
1 & \text{if } l=m\\
0 & \text{if } l\neq m,
\end{cases}
\end{equation*}
hence by \eqref{eq1000} we obtain \eqref{eq909}.\\
Similarly, we can prove \eqref{eq901}.\\
We prove \eqref{eq903} as follows. From \eqref{eq1000}, we have
\begin{eqnarray*}
RHS \text{ }of \text{ }\eqref{eq901} &=&\sum_{j=0}^{n}\sum_{m=0}^{n}\frac{s_{%
\overline{\alpha }}(m,j)s_{\overline{\alpha }}(n,m)}{m!}B_{j,\overline{
\alpha} ,L}^{(k)}(z) \\
&=&\sum_{j=0}^{n}\sum_{m=0}^{n} \frac{s_{\overline{\alpha }%
}(m,j)s_{\overline{\alpha} }(n,m)}{m!}\sum_{i=0}^{j}\sum_{l=i}^{j}l!(-1)^{l+j}\binom{l}{i}S_{\overline{\alpha} }(j,l)\frac{(\ell_{1}\ell_{2}...\ell_{k})^{l-i+1}}{(l-i+1)^{k}}(-z)^{i}\\
&=&\sum_{m=0}^{n}\frac{s_{\overline{\alpha }}(n,m)}{m!}\sum_{l=0}^{n}\sum_{i=0}^{l}l!\binom{l}{i} \frac{(\ell_{1}\ell_{2}...\ell_{k})^{l-i+1}}{(l-i+1)^{k}}(-z)^{i}\sum_{j=l}^{n}(-1^{l+j})s_{\overline{\alpha }}(m,j)S_{\overline{\alpha} }(j,l),
\end{eqnarray*}
since
\begin{equation*}
\sum_{j=l}^{n}(-1)^{j+l}S_{\overline{\alpha }}(j,l)s_{%
\overline{\alpha }}(m,j) = 
\begin{cases}
1 & \text{if } l=m\\
0 & \text{if } l\neq m,
\end{cases}
\end{equation*}
hence by \eqref{eq20} we obtain \eqref{eq903}.\\
Similarly, we can prove \eqref{eq904}.\\
\end{proof}

$\mathbf{References}$
\bibliographystyle{elsarticle-num}

\begin{thebibliography}{99}
\bibitem{Bayad}
A. Bayad,Y.  Hamahata, Polylogarithms and poly- Bernoulli polynomials, Kyushu J. Math., 65(2011) 15-24 .
\bibitem{Cakic}
N. P. Caki\'{c}, The complete Bell polynomial and numbers of Mitrnovi\'{c}, Univ. Beograd. Publ. Elektrotehn. Fak., 6(1995).
\bibitem{Cakic1}
N. P. Caki\'{c}, B.S. El-Desouky and G.V. Milovanovi\'{c}, Explicit formulas and combinatorial identities for generalized Stirling numbers, Mediterr. J. Math., 10(2013) 57-72.
\bibitem{Candelpergher}
B. Candelpergher, M.-A. Coppo, A new class of identites involving Cauchy numbers, harmonic numbers and zeta values, The Ramanujan J., 27(2012) 305-328.
\bibitem{Comtet1}
L. Comtet, Nombers de Stirling generaux et fonctions symetriques, C. R. Acad, Sc. Par. (series A), 275(1972) 747-750.
\bibitem{Comtet2}
L. Comtet, Advanced combinatorics, Reidel, Dordrecht, 1974.
\bibitem{Coppo}
M.-A. Coppo, B. Candelpergher, The Arakawa-Kaneko zeta functions, Ramanujan J., 22(2010) 153-162.
\bibitem{Desouky}
B. S. El-Desouky, The multiparameter non-central Stirling numbers, The Fibonacci Quart., 32(1994) 218-225.
 \bibitem{Gould}
H. W. Gould, Stirling number representation problems, Proc. Amer. Math. Soc., 11(1960) 447-451.
 \bibitem{Kamano}
K.  Kamano and T. Komastu, Poly- Cauchy polynomials, Mosc. J. Comb. Number Theory, 3(2013) 61-87.
 \bibitem{Kaneko}
M. Kaneko, Poly- Bernoulli numbers, J. Th\'{e}or Nombr Bordx., 9(1997) 221-228.
 \bibitem{Komastu1}
T. Komastu, Poly- Cauchy numbers, Kyushu J. Math., 67(2013) 143-153.
 \bibitem{Komastu2}
T. Komastu and  F. Lucab, Some relationships between poly- Cauchy numbers and poly- Bernoulli numbers,  Ann. Math. Inform, 41(2013) 99-105.
\bibitem{Komastu3}
T.  Komastu, Poly- Cauchy numbers with a q parameter, Ramanujan J., 31(2013) 353-371.
\bibitem{Komastu4}
T. Komastu, V. Laohakosd and K. Liptai,  A Generalization of poly- Cauchy numbers and their properties, Abstr. Appl. Anal., 2013, Article ID 179841.
\bibitem{Merlini}
D. Merlini, R. Sprugnoli and M. C. Verri, The Cauchy numbers, Discrete Math., 306(2006) 1906-1920.
\bibitem{Petkovsek}
M. Petkovsek,T. Pisanski, Combinatorial interpretation of unsigned Stirling and Lah numbers, Pi Mu Epsilon J., 12(2007) 417-424.
%
\end{thebibliography}

%
\end{document}